\documentclass[12pt,reqno]{amsart}
\usepackage{bbm}
\usepackage{mathrsfs}
\usepackage{amsfonts}
\usepackage{amsfonts}
\usepackage{amsfonts}
\usepackage{amsfonts}
\usepackage{amsfonts}
\usepackage{amsfonts}
\usepackage{amsfonts}
\usepackage{amsfonts}
\usepackage{amsfonts}
\usepackage{amsfonts}
\usepackage{amsfonts}
\usepackage{amsfonts}
\usepackage{amsfonts}
\usepackage{amsfonts}
\usepackage{amsfonts}
\usepackage{amsfonts}
\usepackage{amsfonts}
\usepackage{amsfonts}
\usepackage{a4wide}
\newtheorem{lemma}{Lemma}[section]

\newtheorem{theorem}{Theorem}[section]

\let\Section=\section
\def\section{\setcounter{equation}{0}\Section}
\sf
\begin{document}
\title[Threshold results for parabolic systems ]
{Threshold results for semilinear parabolic systems }
\author{Qiuyi Dai \ Haiyang He \ Junhui Xie}
\thanks{Department of Mathematics, Hunan Normal University, Changsha Hunan 410081,
 P.R.China}
\thanks{Corresponding author: Haiyang He(hehy917@yahoo.com.cn)}
\thanks{This work was supported by NNSFC(No.10971061)}

\maketitle

\vskip 0.2cm \arraycolsep1.5pt
\newtheorem{Lemma}{Lemma}[section]
\newtheorem{Theorem}{Theorem}[section]
\newtheorem{Definition}{Definition}[section]
\newtheorem{Proposition}{Proposition}[section]
\newtheorem{Remark}{Remark}[section]
\newtheorem{Corollary}{Corollary}[section]

\begin{abstract} In this paper, we study initial boundary value problem of semi-linear parabolic systems
\begin{equation}\label{eq:0.1} \left\{
  \begin{array}{ll}
u_t-\Delta u=v^p \ &  (x,t)\in   \Omega\times(0, T), \\[2mm]
v_t-\Delta v=
u^q  \ &  (x,t)\in   \Omega\times(0, T),\\[2mm]
u(x,t)=v(x, t)=0   \ &  (x,t)\in   \partial\Omega\times[0, T],\\[2mm]
u(x,0)=u_0(x)\geq0     \ \    &   x\in  \Omega,\\[2mm]
v(x,0)=v_0(x)\geq 0     \ \    &   x\in  \Omega
\end{array}
 \right.
 \end{equation}
and prove that any positive solution of its steady-state problem
\begin{equation}\label{eq:0.2} \left\{
  \begin{array}{ll}
-\Delta u=v^p \ &  x\in   \Omega, \\[2mm]
 -\Delta v=
u^q  \ &  x\in   \Omega,\\[2mm]
u=v=0  \ &  x\in   \partial\Omega
\end{array}
 \right.
 \end{equation}
 is an initial datum threshold for the
existence and nonexistence  of global solution to problem(\ref{eq:0.1}).
For the precisely statement of this result, see Theorem 1.1  in the introduction of this paper.

\end{abstract}
\vspace{3mm}
\noindent {\bf AMS} classification:  35J50,  35J60. \vspace{3mm}

\section {Introduction}

\setcounter{equation}{0}

Let $\Omega$ be a bounded domain in $R^N$. We consider the following
initial-boundary value problem
\begin{equation}\label{eq:1.1} \left\{
  \begin{array}{ll}
u_t-\Delta u=v^p \ &  (x,t)\in   \Omega\times(0, T), \\[2mm]
v_t-\Delta v=
u^q  \ &  (x,t)\in   \Omega\times(0, T),\\[2mm]
u(x,t)=v(x, t)=0   \ &  (x,t)\in   \partial\Omega\times[0, T],\\[2mm]
u(x,0)=u_0(x)\geq 0 \ \    &   x\in  \Omega,\\[2mm]
v(x,0)=v_0(x)\geq 0   \ \    &   x\in  \Omega,
\end{array}
 \right.
 \end{equation}
 where $u_t, v_t$ are, respectively, the partial derivatives of $u(x,t)$ and $v(x,t)$ with respect to variable $t$,
 $\Delta=\sum\limits_{i=1}^{N}\frac{\partial^2}{\partial x_i^2}$ is the Laplace
 operator, and $N\geq 2$, $p,q>1$ satisfy
 \begin{equation}\label{eq2}
 \frac{1}{p+1}+\frac{1}{q+1}>\frac{N-2}{N}.
 \end{equation}

 It is well known that for any $u_0(x), v_0(x)\in L^{\infty}(\Omega)$, problem (\ref{eq:1.1}) has an unique classical
 solution $(u(x,t),v(x,t))$ in a short time, which is called a local solution of problem (\ref{eq:1.1}).
 Let $T_{\max}$ denote the maximum existence time of
 $(u(x,t),v(x,t))$ as a classical solution. If $T_{\max}=+\infty$,
 then we say that $(u(x,t),v(x,t))$ exists globally, or problem
 (\ref{eq:1.1}) has global solution. If $T_{\max}<+\infty$, then we
 have
 $$\lim\limits_{t\rightarrow
 T_{\max}^-}\sup\limits_{x\in\Omega}u(x,t)=\lim\limits_{t\rightarrow
 T_{\max}^-}\sup\limits_{x\in\Omega}v(x,t)=+\infty,$$
 for which we say that $(u(x,t),v(x,t))$ blows up in a finite time
 (see for example \cite{QuSoup} for more details).

 It is also well known that the solution $(u(x,t),v(x,t))$ exists
 globally when the initial value $(u_0(x),v_0(x))$ is small enough
 in some sense, and blows up in a finite time when  the initial value $(u_0(x),v_0(x))$
 is large enough in a suitable sense (see
 \cite{CM}\cite{D}\cite{EH1}\cite{EH2}\cite{QuSoup} for the exact statement).
 However, the classification of the initial datum $(u_0(x),v_0(x))$
 according to the existence or nonexistence of global solutions to problem
 (\ref{eq:1.1}) is still far from complete. Hence, an important
 task in the study of problem (\ref{eq:1.1}) is to find exact
 conditions on the initial datum $(u_0(x),v_0(x))$ which can ensure
 the existence or nonexistence of global solutions to problem
 (\ref{eq:1.1}). On this direction, we present here a so called
 threshold result for problem (\ref{eq:1.1}) by making use of its
 positive equilibriums. To state our result
 simply and precisely, we introduce some notations and definitions
 first. For any planar vector $(a,b)$ and $(c,d)$, we use
 $(a,b)\geq(c,d)$ to mean that $a\geq c$ and $b\geq d$, and
 $(a,b)=(c,d)$ to mean that $a=c$ and $b=d$. If $a,b,c,d$ are
 functions of variable $x$, we use
 $(a(x),b(x))\not\equiv(c(x),d(x))$ to mean that there exists at
 least one point $x_0$ such that
 $(a(x_0),b(x_0))\neq(c(x_0),d(x_0))$. Finally, we say that
 $(U(x),V(x))$ is a positive equilibrium of problem (\ref{eq:1.1})
 if it is a solution of the following steady-state problem related to problem
 (\ref{eq:1.1}).

\begin{equation}\label{eq:1.2}
\left\{\begin{array}{ll}
-\Delta U=V^p  &  x\in   \Omega, \\
-\Delta V=U^q  &  x\in   \Omega,\\
(U,V)>(0,0) & x\in\Omega\\
(U,V)=(0,0)  \ &  x\in\partial\Omega.
\end{array}
\right.
\end{equation}

Keeping the above notations and definitions in mind, our main result
of this paper can be stated as
\begin{theorem}\label{tm:1.1}
Assume that $p,q>1$ satisfy (\ref{eq2}), and that $(U(x), V(x))$ is
an arbitrary smooth solution of problem (\ref{eq:1.2}). Then there
holds

\vskip 0.1in

 (i) If $\big(0, 0\big)\leq \big(u_0(x), v_0(x)\big)\leq
\big(U(x), V(x)\big)$ and $\big(u_0(x), v_0(x)\big)\not\equiv
\big(U(x), V(x)\big)$, then problem (\ref{eq:1.1}) has a global
solution $(u(x,t), v(x, t))$. Moreover,
$\lim\limits_{t\rightarrow\infty} (u(x, t),v(x,t))=(0,0)$.

\vskip 0.1in

(ii) If $\big(u_0(x), v_0(x)\big)\geq \big(U(x), V(x)\big)$ and
$\big(u_0(x), v_0(x)\big)\not\equiv \big(U(x), V(x)\big)$, then the
solution $(u(x,t), v(x, t))$ of problem (\ref{eq:1.1}) blows up in a
finite time.
\end{theorem}

We remark here that Theorem 1.1 is a natural generalization of
results on scalar equations proved by P.L.Lions in \cite{Lion} and
A.A.Lacey in \cite{Lac},but the method we use here is different.
Roughly speaking, Theorem 1.1 says that any smooth solution of
problem (\ref{eq:1.2}) is an initial datum threshold for the
existence and nonexistence of global solutions to problem
(\ref{eq:1.1}). It is also worth pointing out that the restriction
(\ref{eq2}) on the exponents $p$ and $q$ is optimal in the sense
that problem (\ref{eq:1.2}) has no solutions for star-shaped domains
$\Omega$ when (\ref{eq2}) is violated (see \cite{CFM}).

The plan of this paper is as follows. Section 2 devotes to prove two
lemmas need in the proof of theorem 1.1. The proof of Theorem 1.1 is
given in Section 3. Some further remarks are included in Section 4.

\section  {Preliminaries}
In this section, we prove two lemmas which will  be used later in
the proof of our main result.

\begin{lemma}\label{lm:3.1}
Let $(g(x), h(x))$ and $(U(x), V(x))$ be two distinct smooth
solutions of problem (\ref{eq:1.4}). Then we have
 \[\int_\Omega g(x) U(x)(g^{q-1}-U^{q-1})\ dx=\int_\Omega h(x)V(x)(V^{p-1}-h^{p-1})\ dx.\]
\end{lemma}

\begin{proof}
 This result can be found  in \cite{DF}.  However, for the reader's convenience, We give a proof here.
Since $(g(x), h(x))$ and $(U(x), V(x))$ are solutions of problem
(\ref{eq:1.2}), we have
\begin{equation}\label{eq:3.2} \left\{
  \begin{array}{ll}
-\Delta g(x)=h^p \ &  x\in   \Omega, \\[2mm]
 -\Delta h(x)=
g^q  \ &  x\in   \Omega,\\[2mm]
g=h=0   \ &  x\in   \partial\Omega
\end{array}
 \right.
 \end{equation}
 and
 \begin{equation}\label{eq:3.3} \left\{
  \begin{array}{ll}
-\Delta U(x)=V^p \ &  x\in   \Omega, \\[2mm]
 -\Delta V(x)=
U^q  \ &  x\in   \Omega,\\[2mm]
U=V=0   \ &  x\in   \partial\Omega.
\end{array}
 \right.
 \end{equation}
 From these, we can derive
 \[\int_\Omega h^p V(x)\ dx=-\int_\Omega \Delta g(x) V(x) \ dx=-\int_\Omega g(x)\Delta V\ dx=\int_\Omega g(x) U^q,\]
 \[\int_\Omega h V^p\ dx=-\int_\Omega \Delta U(x) h(x) \ dx=-\int_\Omega U(x) \Delta h(x)\ dx=\int_\Omega U g^q.\]
Consequently

 \[\int_\Omega g(x)U(x)(g^{q-1}-U^{q-2})\ dx=\int_\Omega
h(x)V(x)(V^{p-1}-h^{p-1})\ dx.\]

\end{proof}

\begin{lemma}\label{lm:2.1}
Assume that $x>0, y>0$, and $0<a<1$. Then \[x^a+y^a\leq
2^{1-a}(x+y)^a.\]
\end{lemma}

\begin{proof}
Let $g(t)=t^a+(1-t)^a$, $0<t<1$. An easy computations yields
\[g'(t)=a t^{a-1}-a(1-t)^{a-1}. \]
Hence, we have
\[ \left\{
  \begin{array}{ll}
g'(t)>0\ \ \  &  \ \  0<t<\frac{1}{2}, \\[2mm]
g'(t)=0 \ \ \ &  \ \  t=\frac{1}{2},\\[2mm]
g'(t)<0   \ \ \ &  \ \ \frac{1}{2}<t<1
\end{array}
 \right.\]
From this, we conclude that

 \[g(t)=t^a+(1-t)^a\leq 2^{1-a}. \]
Substituting $t=\frac{x}{x+y}$ into the above inequality, we finally
obtain that
\[x^a+y^a\leq 2^{1-a}(x+y)^a.\]
\end{proof}

\section  {Proof of Theorem  1.1}

\textbf{Proof of Theorem  1.1:}\  (i) Since $(0,0)\leq(u_0(x),
v_0(x))\leq(U(x), \ V(x))$, and $U(x), V(x)\in L^\infty(\Omega)$, we
know that problem (\ref{eq:1.1}) has a global solution $(u(x, t),
v(x, t))$. Noticing that $(u_0(x), v_0(x))\not\equiv(U(x), V(x))$,
it follows from the maximum principle and the strong comparison
principle that
\[(0,0)\leq(u(x, t), v(x, t))<(U(x),
V(x)\] for any $(x,t)\in \Omega\times (0, +\infty)$.

Therefore, we may assume, by replacing $(u_0(x), v_0(x))$ with
$(u(x, T), v(x, T))$ for some $T>0$ if necessary, that $(u_0(x),
v_0(x))\leq(\alpha U(x), \alpha V(x))$ for some constant
$0<\alpha<1$. Let $(g_\alpha (x), h_\alpha(x))=(\alpha U(x), \alpha
V(x))$. It is easy to verify that $(g_\alpha (x), h_\alpha(x))$
satisfies
 \begin{equation}\label{eq:3.4} \left\{
  \begin{array}{ll}
-\Delta g_\alpha >(h_\alpha)^p \ &  x\in   \Omega, \\[2mm]
 -\Delta h_\alpha>
(g_\alpha )^q  \ &  x\in   \Omega,\\[2mm]
 g_\alpha=h_\alpha=0   \ &  x\in   \partial\Omega,\\[2mm]
g_\alpha >0, h_\alpha >0\ &  x\in   \Omega.
\end{array}
 \right.
 \end{equation}
This implies that $(g_\alpha (x), h_\alpha(x))$ is a strict
super-solution of the following problem
\begin{equation}\label{eq:3.5} \left\{
\begin{array}{ll}
G_t-\Delta G=H^p \ &  (x,t)\in   \Omega\times(0, T), \\[2mm]
H_t-\Delta H=
G^q  \ &  (x,t)\in   \Omega\times(0, T),\\[2mm]
G=H=0   \ &  (x,t)\in   \partial\Omega\times[0, T],\\[2mm]
G(x,0)=g_\alpha \geq0      \ \    &   x\in  \Omega,\\[2mm]
H(x,0)=h_\alpha\geq 0 \ \    &   x\in  \Omega.
\end{array}
\right.
\end{equation}
Let $(G(x,t), H(x, t))$ be the solution of (\ref{eq:3.5}). By strong
comparison principle we know that $(G(x, t), H(x, t))$ is strictly
decreasing with respect to $t$, and $(0,0)\leq(G(x, t),
H(x,t))\leq(U(x), V(x))$. Therefore, $(G(x, t), H(x, t))$ exists
globally. Moreover, there are some functions $g(x)$ and $h(x)$ such
that
 \[\lim\limits_{t\rightarrow\infty} G(x,t)=g(x), \quad  \lim\limits_{t\rightarrow\infty} H(x,t)=h(x).\]
uniformly on $\Omega$, and $(g(x), h(x))$ is a smooth solution of
the following problem.
\begin{equation}\label{eq:1.4}
\left\{\begin{array}{ll}
-\Delta g=h^p  &  x\in   \Omega, \\
-\Delta h=g^q  &  x\in   \Omega,\\
(g,h)=(0,0)  \ &  x\in\partial\Omega.
\end{array}
\right.
\end{equation}
From this, we conclude that $(g(x), h(x))\equiv(0,0)$. Otherwise, by
strong maximum principle, we have $(g(x), h(x))>(0,0)$. On the other
hand, we have $(g(x),h(x))<(U(x),V(x))$ since $(G(x, t), H(x, t))$
is strictly decreasing with respect to $t$. Thus
\[\int_\Omega g(x) U(x)(g^{q-1}-U^{q-1})\ dx<0, \quad \int_\Omega h(x)V(x)(V^{p-1}-h^{p-1})\ dx>0.\]
This is a contradiction with Lemma 2.1. Therefore
 \[\lim\limits_{t\rightarrow\infty} (G(x,t),H(x,t))=(0,0).\]
Noticing that $(0, 0)\leq(u_0, v_0) \leq( g_\alpha(x), h_\alpha(x))
$, comparison principle ensures
 \[(0,0)\leq (u(x, t),v(x,t))\leq (G(x, t),H(x, t)).\]
By applying squeeze principle, we obtain
 \[\lim\limits_{t\rightarrow\infty}(u(x, t),v(x,t))=(0,0).\]

\vskip 0.1in

(ii)\ we prove the conclusion (ii) of Theorem 1.1 by contradiction.
To this end, we assume that $(u_0(x), v_0(x))\geq(U(x), V(x))$,
$(u_0(x), v_0(x)) \not \equiv( U(x), V(x))$ and problem
(\ref{eq:1.1}) has a global solution $(u(x, t), v(x,t))$. By strong
comparison principle, we have
  \[(u(x, t), v(x,t)) > (U(x), V(x)), \]
for  any  $(x, t)\in \bar{\Omega}\times (0,+\infty)$. Therefore, we
may assume, by replacing $(u_0(x), v_0(x))$ with $(u(x, T), v(x, T;
u_0, v_0))$ for some $T>0$  if necessary, that $(u_0(x), v_0(x))\geq
(\beta U(x),\beta V(x))$ for some constant $\beta>1$. Let $(g_\beta,
h_\beta)=(\beta U(x), \beta V(x))$. It is easy to verify that
$(g_\beta, h_\beta)$ satisfies
 \begin{equation}\label{eq:3.7} \left\{
  \begin{array}{ll}
-\Delta g_\beta<(h_\beta)^p \ &  x\in   \Omega, \\[2mm]
 -\Delta h_\beta<
(g_\beta)^q  \ &  x\in  \Omega,\\[2mm]
 g_\beta=h_\beta=0   \ &  x\in   \partial\Omega.
\end{array}
 \right.
 \end{equation}
Hence, $(g_\beta, h_\beta)$ is a strict sub-solution of the
following problem
\begin{equation}\label{eq:3.8} \left\{
\begin{array}{ll}
G_t-\Delta G=H^p \ &  (x,t)\in   \Omega\times(0, T), \\[2mm]
 H_t-\Delta H=
G^q  \ &  (x,t)\in   \Omega\times(0, T),\\[2mm]
G=H=0   \ &  (x,t)\in   \partial\Omega\times[0, T],\\[2mm]
G(x,0)=g_\beta \geq0    \ \    &   x\in  \Omega,\\[2mm]
H(x,0)=h_\beta\geq 0     \ \    &   x\in  \Omega.
\end{array}
 \right.
 \end{equation}
 Let $(G(x, t), H(x, t))$ be the solution of problem (\ref{eq:3.8}).
 Then it follows from the comparison principle that
\[(G(x, t), H(x, t))\leq (u(x, t), v(x, t))\] for any $(x, t)$ due to
 $(g_\beta (x), h_\beta(x))\leq \big(u_0,  v_0\big)$.
 Consequently, $(G(x, t), H(x, t))$ exists globally and is strictly increasing
 with respect to $t$.

 Let

 \[\varphi(t)=\int_\Omega G(x, t)H(x,t)\ dx.\]
 \[E(t)=\int_\Omega \nabla G\nabla H\ dx-\frac{1}{p+1}\int_\Omega H^{p+1}\ dx-\frac{1}{q+1}\int_\Omega G^{q+1}\ dx,\]

 By making use of (\ref{eq:3.8}), we can verify that $\varphi(t)$ and $E(t)$ satisfy
 \[\dfrac{d \varphi}{dt}=-2E(t)+\frac{p-1}{p+1}\int_\Omega H^{p+1}\ dx+\frac{q-1}{q+1}\int_\Omega G^{q+1}\ dx,\]
 \[\begin{array}{ll}
 \dfrac{d E(t)}{dt}&=\int_\Omega \nabla G_t \nabla H\ dx+\int_\Omega \nabla H_t \nabla G\ dx
 -\int_\Omega H^p H_t\ dx-\int_\Omega G^q G_t\ dx \\[5mm]
 &=-2\int_\Omega G_t H_t\ dx\leq 0.
 \end{array}\]
Let $\gamma=\frac{(p+1)(q+1)}{p+q+2}$. It follows from the
assumption $p>1$ and $q>1$ that
\[\gamma>1, \ \  \frac{q+1}{\gamma}>1, \ \  \frac{p+1}{\gamma}>1, \ \ \frac{\gamma}{q+1}+\frac{\gamma}{p+1}=1.\]
By H\"{o}lder's inequality, Young's inequality, and Lemma 2.2, we
have
 \[\begin{array}{ll}
 \varphi(t)&\leq \frac{q+1}{\gamma}\int_\Omega G^\frac{q+1}{\gamma}\ dx
 +\frac{p+1}{\gamma}\int_\Omega H^\frac{p+1}{\gamma}\ dx\\[5mm]
 &\leq \frac{\max\{p,q\}+1}{\gamma}|\Omega|^{1-\frac{1}{\gamma}}\big((\int_\Omega G^{q+1}\ dx)^\frac{1}{\gamma}
 +(\int_\Omega H^{p+1}\ dx)^\frac{1}{\gamma}\big)\\[5mm]
 &\leq \frac{\max\{p,q\}+1}{\gamma}|\Omega|^{1-\frac{1}{\gamma}} 2^{1-\frac{1}{\gamma}}[\int_\Omega G^{q+1}\ dx+\int_\Omega H^{p+1}\ dx]^\frac{1}{\gamma}.
 \end{array}\]
 Hence, there exists a positive constant $C$ such that
 \[\dfrac{d \varphi}{dt}\geq -2E(t)+C\varphi^{\gamma}(t).\]
 Since $E(t)$ is decreasing in $t$, we have $E(t)\leq E(0)$ for any $t>0$. Consequently,
  \[\dfrac{d \varphi}{dt}\geq -2E(0)+C\varphi^{\gamma}(t).\]
From this, we may conclude that
\[\sup\limits_{t\geq0}\int_\Omega
GH\ dx<+\infty .\] Otherwise, we have  $\int_\Omega GH\
dx\rightarrow +\infty, \ as \ t\rightarrow\infty$ due to
$\int_\Omega GH\ dx=\varphi(t)$ is strictly increasing in $t$.
Hence, there exists a constant $T>0$ large enough such that
\[\frac{d}{dt}\int_\Omega GH\ dx\geq \frac{C}{2}(\int_\Omega GH\ dx)^{\gamma},\]
for any $t>T$. This implies that $(G(x,t),H(x,t))$ must blow up in a
finite time which contradicts the fact that $(G(x,t),H(x,t))$ is a
global solution of problem (\ref{eq:3.8}).

Let \[T(t)=\int_\Omega G^{q+1}\ dx+\int_\Omega H^{p+1}\ dx.\] Then,
$T(t)$ is strictly increasing in $t$ because $(G(x,t),H(x,t))$ does.
Thus, for any $t>0$, we have
 \[\begin{array}{ll}
 C\geq \int_t^{t+1}\frac{d}{d s}\int_\Omega GH \ dx \ ds&=-2\int_t^{t+1} E(s)
 ds
 +\frac{p-1}{p+1} \int_t^{t+1} \int_\Omega G^{p+1}\ dx \ ds \\[5mm]
 &\ \ \ \ +\frac{q-1}{q+1}\int_t^{t+1} \int_\Omega H^{q+1}\ dx \ ds,\\[5mm]
 &\geq  -2 E(0)+\min\{\frac{p-1}{p+1}, \frac{q-1}{q+1}\}T(t).
 \end{array}\]
From this, we can easily see that

\[\sup\limits_{t\geq 0}T(t)<+\infty.\]

Consequently, there are functions $g(x)\in L^{p+1}(\Omega)$ and $
h(x)\in L^{q+1}(\Omega)$ such that
  \[G(x, t)\rightarrow g(x)\ \    \mbox{weakly in}\ \ \  L^{p+1}(\Omega), \]
\[H(x, t)\rightarrow h(x)\ \ \  \mbox{weakly in} \ \ \ L^{q+1}(\Omega).\]
Multiplying the first and the second equation in (\ref{eq:3.8}) by
$\varphi$ and $\psi$ respectively, and integrating the result
equations on $[t,t+1]$, we obtain
\[\int_\Omega [G(x, t+1)-G(x, t)]\varphi  \ dx \ ds
 +\int_t^{t+1}\int_\Omega G(-\Delta \varphi )\ dx \ ds= \int_t^{t+1}\int_\Omega H^p\varphi \ dx \ ds,\]
 \[\int_\Omega[ H(x, t+1)-H(x, t)]\psi  \ dx \ ds
 +\int_t^{t+1}\int_\Omega H(-\Delta \psi )\ dx \ ds= \int_t^{t+1}\int_\Omega G^p\psi \ dx \ ds,\]
 Passing to the limit as $t\rightarrow \infty$, we find that
\[\int_\Omega g(-\Delta \varphi )\ dx = \int_\Omega h^p\varphi \ dx ,\]
 \[\int_\Omega h(-\Delta \psi )\ dx = \int_\Omega g^p\psi \ dx .\]
This implies that $(g(x), h(x))$ is a $L^1$ solution of problem
(\ref{eq:1.2}) (For the definition of the $L^1$ solution, we refer
to \cite{QuSoup}).

Noticing that $p,q>1$ satisfy (\ref{eq2}) and
\[\int_\Omega g^{q+1}\ dx<+\infty, \quad  \int_\Omega
h^{p+1}\ dx<+\infty,\] it follows from the regularity theory
(bootstrap method) of $L^1$ solution that $g, h\in L^\infty(\Omega)$
(see \cite{QuSoup}). With $L^{\infty}$ estimate in hand, we can
establish the $H_0^1$ estimate of $g(x)$ and $h(x)$ by making use of
the following facts
\begin{equation}\label{eq:3.9} \left\{
\begin{array}{ll}
-\Delta G\leq H^p \quad & (x,t)\in\Omega\times(0,+\infty), \\[2mm]
-\Delta H\leq G^q  \quad &  (x,t)\in\Omega\times(0, +\infty),\\[2mm]
G=H=0   \quad &  (x,t)\in\partial\Omega\times[0, +\infty).
\end{array}
\right.
\end{equation}
Now, we can conclude that $(g, h)$ is a classical solution of
problem (\ref{eq:1.2}) by the standard regularity theory of elliptic
differential equations (see \cite{GT}).

Since $( g_\beta(x), h_\beta(x))>(U(x),V(x))$,
 it follows from the strong comparison principle that \[ G(x, t)>U(x), \quad  H(x, t)>V(x)\] for any
 $(x, t)$. Consequently \[g(x)>U(x), \quad h(x)>V(x).\]
From this, we have
 \[\int_\Omega g(x) U(x)(g^{q-1}-U^{q-1})\ dx>0\ \ \ \mbox{and}\ \ \
\int_\Omega h(x)V(x)(V^{p-1}-h^{p-1})\ dx<0.\] This is a
contradiction with the conclusion of Lemma 2.1 and we complete the
proof of Theorem 1.1 (ii).
 $\Box$

\section  {Further Remarks}

The method used in the proof of theorem 1.1 can be applied to study
the following inhomogeneous problem
\begin{equation}\label{eq:4.1}
\left\{\begin{array}{ll}
u_t-\Delta u=v^p+\lambda f(x) \ &  (x,t)\in   \Omega\times(0, T), \\[2mm]
v_t-\Delta v=u^q+\lambda g(x)  \ &  (x,t)\in   \Omega\times(0, T),\\[2mm]
(u,v)=(0,0)   \ &  (x,t)\in   \partial\Omega\times[0, T],\\[2mm]
(u(x,0),v(x,0))=(u_0(x),v_0(x))\geq(0,0) \ \    &   x\in  \Omega,
\end{array}
 \right.
 \end{equation}
where $p,q>1$ satisfy (\ref{eq2}), and
$(0,0)\leq(f(x),g(x))\not\equiv(0,0)$.

The main difference between problem (\ref{eq:1.1}) and
(\ref{eq:4.1}) lies in the structure of their equilibrium sets. From
lemma 2.1, we can easily see that any two distinct equilibriums of
problem (\ref{eq:1.1}) must intersect. However, problem
(\ref{eq:4.1}) has an unique minimal equilibrium for $\lambda>0$
small enough which separates from other equilibriums. To state our
results precisely, we consider the following steady-state problem of
problem (\ref{eq:4.1})
\begin{equation}\label{eq:4.2}
\left\{\begin{array}{ll}
-\Delta u=v^p+\lambda f(x) \ &  x\in   \Omega, \\[2mm]
-\Delta v=u^q+\lambda g(x)  \ & x\in   \Omega,\\[2mm]
(u,v)>(0,0)\ & x\in\Omega,\\
(u,v)=(0,0)   \ &  x\in\partial\Omega.
\end{array}
\right.
\end{equation}
By sub-solution and sup-solution method, it is not difficult to
prove the following

\begin{lemma}\label{lm:4.1}
There exists a positive number $\lambda^*$ such that the following
two statements are true.

\vskip 0.08in

(i)\ If $\lambda>\lambda^*$, then problem (\ref{eq:4.2}) has no
solution.

\vskip 0.08in

(ii)\ If $0<\lambda<\lambda^*$, then problem (\ref{eq:4.2}) has an
unique minimal solution $(u_{\min}(x),v_{\min}(x))$ in the sense
that $((u_{\min}(x),v_{\min}(x))\leq(u(x),v(x))$ for any solution
$(u(x),v(x))$ of problem (\ref{eq:4.2}). Moreover, if
$(u(x),v(x))\not\equiv(u_{\min}(x),v_{\min}(x))$, then
$((u_{\min}(x),v_{\min}(x))<(u(x),v(x))$.
\end{lemma}

Let $u(x)=U(x)+u_{\min}(x)$ and $v(x)=V(x)+v_{\min}(x)$. Then, it is
easy to see that $(U,V)$ satisfies

\begin{equation}\label{eq:4.3}
\left\{\begin{array}{ll}
-\Delta U=(V+v_{\min})^p-v^p_{\min}  &  x\in   \Omega, \\
-\Delta V=(U+u_{\min})^q-u^q_{\min} &  x\in   \Omega,\\
(U,V)=(0,0)  \ &  x\in\partial\Omega.
\end{array}
\right.
\end{equation}
By variational method, we can prove that problem (\ref{eq:4.3}) has
at least one positive solution provided that (\ref{eq2}) holds (see
\cite{HanL}). Hence, we have

\begin{theorem}\label{tm:4.2}
Assume that $p,q>1$ satisfy (\ref{eq2}). Let $\lambda^*$ be the
number obtained in lemma 4.1. Then, for any
$\lambda\in(0,\lambda^*)$, problem (\ref{eq:4.2}) has at least two
solutions, and among them there exists a minimal one.
\end{theorem}

By the same method as that used in the proof of lemma 2.1, we can
prove the following

\begin{lemma}\label{lm:4.3}
Let $(U_1, V_1)$ and $(U_2, V_2)$ be any two smooth solutions of
problem (\ref{eq:4.3}),
$G(u)=\frac{(u+u_{\min})^{q}-u^{q}_{\min}}{u}$ and
$H(v)=\frac{(v+v_{\min})^{p}-v^{p}_{\min}}{v}$. Then we have
 \[\int_\Omega U_1U_2(G(U_2)-G(U_1))\ dx
 =\int_\Omega V_1V_2(H(V_1)-H(V_2))\ dx.\]
\end{lemma}
Noting that $G(u)$ and $H(v)$ are strictly increasing in $u$ and $v$
respectively due to $p,q>1$, we infer from lemma 4.2 that the
following result on the structure of solution set of problem
(\ref{eq:4.2}) holds

\begin{theorem}\label{tm:4.4}
With the same assumption as that of theorem 4.1, problem
(\ref{eq:4.2}) has at least two solutions, and among them there
exists a minimal one. Moreover, any two distinct solutions of
problem (\ref{eq:4.2}) which are also different from the minimal one
must intersect somewhere.
\end{theorem}

With theorem 4.2 established, by a similar argument to that used in
the proof of theorem 1.1, we can reach the following
\begin{theorem}\label{tm:4.3}
Assume that $p,q>1$ satisfy (\ref{eq2}). Let $\lambda^*$ be the
number obtained in lemma 4.1. Then, we have

\vskip 0.08in

(i)\ If $\lambda>\lambda^*$, then, for any initial value
$(u_0(x),v_0(x))\geq(0,0)$, the solution $(u(x,t),v(x,t))$ of
problem (\ref{eq:4.1}) must blow up in a finite time.

\vskip 0.08in

(ii)\ If $0<\lambda<\lambda^*$, and $(U(x), V(x))$ is an arbitrary
smooth solution of problem(\ref{eq:4.2}) which is different from the
minimal one, then problem (\ref{eq:4.1}) has a global solution
$(u(x,t), v(x, t))$ with $\lim\limits_{t\rightarrow\infty} (u(x,
t),v(x,t))=(u_{\min}(x),v_{\min}(x))$ provided that $(0, 0)\leq
(u_0(x), v_0(x))\leq(U(x), V(x))$ and $(u_0(x), v_0(x))\not\equiv
(U(x), V(x))$; whereas, the solution $(u(x,t),v(x,t))$ of problem
(\ref{eq:4.1}) must blow up in a finite time if $(u_0(x),
v_0(x))\geq(U(x), V(x))$ and $(u_0(x), v_0(x))\not\equiv(U(x),
V(x))$.
\end{theorem}

Finally, we point out that the method of this paper can also be
applied to study the following initial-boundary value problem with
Robin boundary conditions.
\begin{equation}\label{eq:4.5}
\left\{\begin{array}{ll}
u_t-\Delta u=v^p \ &  (x,t)\in   \Omega\times(0, T), \\[2mm]
v_t-\Delta v=u^q \ &  (x,t)\in   \Omega\times(0, T),\\[2mm]
\frac{\partial}{\partial n}(u,v)+\beta(u,v)=(0,0)   \ &  (x,t)\in   \partial\Omega\times[0, T],\\[2mm]
(u(x,0),v(x,0))=(u_0(x),v_0(x))\geq(0,0) \ \    &   x\in  \Omega,
\end{array}
 \right.
 \end{equation}
where $n$ is the outer unit vector normal to the boundary
$\partial\Omega$ of $\Omega$, and $\beta$ is a positive constant.

By similar arguments to that used in the proof of theorem 1.1, we
can also prove the following result.
\begin{theorem}\label{tm:4.6}
Assume that $p,q>1$ satisfy (\ref{eq2}), and that $(U(x), V(x))$ is
an arbitrary smooth positive equilibrium of problem (\ref{eq:4.5}).
Then there holds

\vskip 0.1in

 (i) If $(0, 0)\leq(u_0(x), v_0(x))\leq
(U(x), V(x))$ and $(u_0(x), v_0(x))\not\equiv(U(x), V(x))$, then
problem (\ref{eq:4.5}) has a global solution $(u(x,t), v(x, t))$.
Moreover, $\lim\limits_{t\rightarrow\infty} (u(x, t),v(x,t))=(0,0)$.

\vskip 0.1in

(ii) If $(u_0(x), v_0(x))\geq(U(x), V(x))$ and $(u_0(x),
v_0(x))\not\equiv(U(x), V(x))$, then the solution $(u(x,t), v(x,
t))$ of problem (\ref{eq:4.5}) must blows up in a finite time.
\end{theorem}

\end{document}